\documentclass[a4paper,12pt, oneside]{amsart}

\usepackage{amsthm, amsmath}
\usepackage{amssymb,amsfonts,amscd,verbatim, longtable}
\usepackage{tikz}
\usepackage{xy}
\usepackage{mathrsfs}
\usepackage{wrapfig}

\usepackage{blkarray}
\usepackage{multirow}
\xyoption{all}
\SelectTips{cm}{10}

\usepackage{ifpdf}
\ifpdf
\usepackage{hyperref}
\else
\usepackage[hypertex]{hyperref}
\fi

\usepackage{multind}
\makeatletter
\renewcommand{\printindex}[2]{\@restonecoltrue\if@twocolumn\@restonecolfalse\fi
  \columnseprule \z@ \columnsep 35pt
  \newpage \twocolumn[{\Large\bf #2 \vskip4ex}]
  \addcontentsline{toc}{chapter}{#2}
  \@input{#1.ind}}
\makeatother

\newtheorem{tr}{Theorem}[section]
\newtheorem*{tr*}{Theorem}

\newtheorem{lemma}[tr]{Lemma}
\newtheorem{pr}[tr]{Proposition}
\newtheorem*{pr*}{Proposition}

\newtheorem{df}[tr]{Definition}
\newtheorem*{df*}{Definition}
\newtheorem*{not*}{Notation}

\theoremstyle{remark}
\newtheorem{rem}[tr]{Remark}

\newlength{\myevenmargin}
\def\differential{d}
\renewcommand\d\differential

\DeclareMathOperator\im{Im}

\DeclareMathOperator\ord{ord}

\DeclareMathOperator\Aut{Aut}

\DeclareMathOperator\codim{codim}

\DeclareMathOperator\id{id}
\DeclareMathOperator\GL{GL}

\DeclareMathOperator\Bl{Bl}
\DeclareMathOperator\Sing{Sing}

\DeclareMathOperator\Tsp{T}

\DeclareMathOperator\Supp{{\rm Supp}}

\def\k{\Bbbk}
\renewcommand\Im\im
\def\bb#1{\mathbb #1}
\def\cal#1{\mathcal #1}
\def\fcal#1{\boldsymbol{\cal{#1}}}

\def\frak#1{\mathfrak{#1}}

\def\ra{\rightarrow}
\def\xra{\xrightarrow}

\def\mto{\mapsto}

\def\pmat#1{\begin{pmatrix}#1\end{pmatrix}}

\def\smat#1{\left(\begin{smallmatrix}#1\end{smallmatrix}\right)}

\def\point#1{\langle #1 \rangle}
\def\UC#1{C_{#1}}

\def\refeq#1{$(\ref{#1})$}
\def\pt{pt}

\def\cTor{{\mathscr T\!\!\mathit{or} } }

\def\P{\bb P}

\def\O{\cal O}

\let\le\leq
\let\ge\geq
\let\star *
\let\subset\subseteq
\let\supset\supseteq

\def\defeq{:=}
\def\eqdef{=:}

\def\iso{\cong}

\def\tilde{\widetilde}

\let\emph\textbf

\newtheorem*{reftR}{Theorem~\refNUMBER}

\newtheorem*{refcoR}{Corollary~\refNUMBER}

\newtheorem*{refpR}{Proposition~\refNUMBER}

\title[On the singular sheaves in the fine Simpson moduli spaces]{On the singular sheaves in the fine Simpson moduli spaces of $1$-dimensional sheaves}
\author{Oleksandr Iena}
\address{University of Luxembourg, Campus Kirchberg\\
Mathematics Research Unit\\
6, rue Richard Coudenhove-Kalergi\\
L-1359 Luxembourg City\\
Grand Duchy of Luxembourg}
\email{oleksandr.iena@uni.lu}
\author{Alain Leytem}
\address{University of Luxembourg, Campus Kirchberg\\
Mathematics Research Unit\\
6, rue Richard Coudenhove-Kalergi\\
L-1359 Luxembourg City\\
Grand Duchy of Luxembourg}
\email{alain.leytem@uni.lu}

\date{}

\begin{document}
\begin{abstract}
In the Simpson moduli space $M$ of semi-stable sheaves  with Hilbert polynomial $dm-1$ on a projective plane we study the closed subvariety $M'$ of sheaves that are not locally free on their support. We show that for $d\ge 4$ it is a  singular subvariety of codimension $2$ in $M$. The blow up of $M$ along $M'$ is interpreted as a (partial) modification of $M\setminus M'$ by vector bundles (on support).
\end{abstract}
\maketitle
\tableofcontents

\section{Introduction}
Let $\k$ be an algebraically closed field of characteristic zero, let $V$ be a vector space over $\k$ of dimension $3$, and let $\P_2=\P V$ be the corresponding projective plane. Consider a linear polynomial with integer coefficients  $P(m)=dm+c\in \bb Z[m]$, $d\in \bb Z_{>0}$, $c\in \bb Z$.
Let $M\defeq M_P(\P_2)$ be the Simpson coarse moduli space (cf.~\cite{Simpson1}) of semi-stable sheaves on $\P_2$ with Hilbert polynomial $P$. As shown in~\cite{LePotier}, $M$ is an irreducible locally factorial variety of dimension $d^2+1$, smooth if $\gcd(d, c)=1$.
\subsection*{Singular sheaves}
The sheaves from $M$ are torsion sheaves on $\P_2$, they are supported on curves of degree $d$.
Restricted to their (Fitting) support most of sheaves in $M$ are vector bundles on curves. Such sheaves constitute an open subvariety $M_B$ of $M$. Its complement $M'$, the closed subvariety of sheaves that are not locally free on their support, is in general non-empty.
This way one could consider $M$ as a compactification of $M_B$. We call the sheaves
from the boundary $M'=M\smallsetminus M_B$ \textit{singular}.

The boundary $M'$ does not have the minimal codimension in general. Loosely speaking, one glues together too many different directions at infinity. For example, for $c\in \bb Z$ with $\gcd(3, c)=1$, all moduli spaces $M_{3m+c}$ are isomorphic to the universal plane cubic curve and $M'_{3m+c}$ is a smooth subvariety of codimension $2$  isomorphic to the universal singular locus of a cubic curve (cf.~\cite{IenaUnivCurve}).

\subsection*{The main result}
As demonstrated in~\cite{IenaCodim2},
for $P(m)=4m+c$ with $\gcd(4, c)=1$, the subvariety $M'$ is singular  of codimension $2$.
The main result of this paper is the following generalization of~\cite{IenaCodim2}.
\begin{tr}\label{tr: main}
For an integer $d\ge 4$, let $M=M_{dm-1}(\P_2)$ be the Simpson moduli space of (semi-)stable sheaves on $\P_2$ with Hilbert polynomial $dm-1$. Let $M'\subset M$ be the subvariety of singular sheaves. Then $M'$ is singular of codimension $2$.
\end{tr}

Using our understanding of $M'$ and the construction from~\cite{IenaUnivCurve} we obtain as a consequence Theorem~\ref{tr: blow-up modification}, which allows to interpret the blow up of $M$ along $M'$ as a modification of the boundary $M'$ by a divisor consisting generically of vector bundles.

\subsection*{Structure of the paper}
In Section~\ref{section: basic constructions}
we recall the results from~\cite{MaicanTwoSemiSt} and identify the open subvariety $M_0$ in $M$ of sheaves without global sections with an open subvariety of  a projective bundle $\bb B$ over a variety of Kronecker modules $N$.
In Section~\ref{section: Singular sheaves},
using a convenient characterization of free ideals of fat curvilinear points on planar curves from Appendix~\ref{section: commutative algebra},
we show that a generic fibre of $M_0'=M_0\cap M'$ over $N$ is 
a union of projective subspaces of codimension $2$ in the fibre of $\bb B$, thus it is
singular of codimension two,
which allows to prove the main result.
In Section~\ref{section: modifications} we briefly discuss how our analysis can be used for modifying the boundary $M'$ by vector bundles (on support).

\section{Basic constructions}\label{section: basic constructions}
\subsection{Kronecker modules}\label{subsection: Kronecker modules}
Let $\bb V$ be the affine space of Kronecker modules
\begin{equation}\label{eq:Kronecker}
(n-1)\O_{\P_2}(-1)\xra{\Phi} n\O_{\P_2}.
\end{equation}
There is a natural group action of $G=(\GL_{n-1}(\k)\times\GL_n(\k))/\k^*$ on $\bb V$.
Since $\gcd(n-1, n)=1$, all semistable points of this action are stable and $G$ acts freely on the open on the open subset of stable points $\bb V^s$.
 Then $\Phi\in \bb V^s$ if and only if $\Phi$ does not lie in the same orbit with a Kronecker module with a zero block of size $j\times(n-j)$, $1\le j\le n-1$ (cf.~\cite[Proposition~15]{Drezet},~\cite[6.2]{EllStro}). In particular, Kronecker modules with linear independent maximal minors are stable.
There exists a geometric quotient $N=N(3; n-1, n)=\bb V^s/G$, which is a smooth projective variety of dimension $(n-1)n$. For more details consult~\cite[Section 6]{EllStro} and~\cite[Section~III]{Drezet}.

The cokernel of a stable Kronecker module $\Phi\in \bb V^s$ is an ideal of a zero-dimensional scheme $Z$ of length $(n-1)n/2$ if the maximal minors of $\Phi$ are coprime. In this case the maximal minors $d_0,\dots, d_{n-1}$ are linear independent
and there is a resolution
\begin{equation}\label{eq:sequence for points ideal}
0\ra(n-1)\O_{\P_2}(-n)\xra{\Phi} n\O_{\P_2}(-n+1)\xra{\pmat{d_0\\ \vdots\\d_{n-1}}} \O_{\P_2}\ra \O_{Z}\ra 0.
\end{equation}
Moreover $Z$ does not lie on a curve of degree $n-2$.
Let $\bb V_0$ denote the open subvariety of $\Phi\in \bb V^s$ of Kronecker modules with coprime maximal minors. Let $N_0\subset N$ be the corresponding open subvariety in the quotient space.

This way one obtains a morphism from $N_0\subset N$  to the Hilbert scheme of zero-dimensional subschemes of $\P_2$ of length $l=(n-1)n/2$, which sends a class of $\Phi\in \bb V^s$ to the zero scheme of its maximal minors. Since every zero dimensional scheme of length $l$ that does not lie on a curve of degree $n-2$ has a minimal resolution of type~\refeq{eq:sequence for points ideal}, this gives an isomorphism of $N_0$ and the open subvariety $H_0\subset H=\P_2^{[l]}$ consisting of $Z$ that do not lie on a curve of degree $n-2$. The complement of $H_0$ is an irreducible hypersurface (as mentioned e.~g. in~\cite[p.~46]{DrezetAltern}).

\subsection{Projective bundle over $N$}\label{subsection: bundle over N}
Let $\bb U_1=(n-1)\Gamma(\P_2, \O_{\P_2}(1))$, $\bb U_2=n \Gamma(\P_2, \O_{\P_2}(2))$.
Consider the trivial vector bundles $\bb V\times \bb U_1$ and $\bb V\times \bb U_2$ over $\bb V$.
Consider the following morphism between them:
\[
\bb V\times \bb U_1\xra{F} \bb V\times \bb U_2, \quad (\Phi, L)\mto (\Phi, L\cdot \Phi).
\]
\begin{lemma}
The morphism $F$ is injective over $\bb V^s$.
\end{lemma}
\begin{proof}
Let  $\Phi\in \bb V^s$ and assume $L\cdot \Phi=0$ for some non-zero $0\neq L=(l_1,\dots, l_{n-1})\in \bb U_1$.

Since $l_i\in \Gamma(\P_2, \O_{\P_2}(1))$,
the dimension of the vector space generated by $\{l_i\}_i$ is at most $3$.
So for some $B\in \GL_{n-1}(\k)$ we get $L\cdot B=L'=(l_1', l_2', l_3', 0,\dots, 0)$ such that the first non-zero entries are linear independent. Then, since $\Phi$ is semistable if and only if $B^{-1}\Phi$ is semistable  and since $L\Phi=(LB)\cdot (B^{-1}\Phi)=0$, we may assume without loss of generality that $L=(l_1, l_2, l_3, 0,\dots, 0)$ and  the first non-zero entries of $L$ are linear independent.

If $l_1\neq 0$, $l_2=l_3=0$, then $L\Phi=0$ implies that the first row of $\Phi$ is zero, which contradicts the stability of $\Phi$.

If $l_1\neq 0$, $l_2\neq 0$, $l_3=0$, then the syzygy module of $(l_1, l_2)$ is generated by $\smat{l_2\\-l_1}$.
So the columns of the first two rows of $\Phi$ are scalar multiples of $\smat{l_2\\-l_1}$ and hence, after performing elementary transformations on the columns of $\Phi$, $\Phi$ is equivalent to a matrix with a zero block of size $2\times(n-1)$, which  contradicts the stability of $\Phi$.

If $l_i\neq 0$, $i=1,2,3$, then the syzygy module of $(l_1, l_2, l_3)$ is generated by $3$ linear independent generators
\[
\pmat{0\\l_3\\-l_2}, \pmat{-l_3\\0\\l_1}, \pmat{l_2\\-l_1\\0},
\]
which implies that $\Phi$ is equivalent to a matrix with a zero block of size $3\times (n-3)$, which contradicts the stability of $\Phi$.
\end{proof}
Therefore, $\bb V^s\times \bb U_1\xra{F}\bb V^s\times \bb U_2$ is a vector subbundle and  hence the cokernel of $F$ is a vector bundle of rank $6\cdot n-3\cdot(n-1)=3n+3=3d$ denoted by $E$.

The group action of $\GL_{n-1}(\k)\times\GL_n(\k)$ on $\bb V^s\times \bb U_1$ and $\bb V^s\times \bb U_2$ induces a group action of $\GL_{n-1}(\k)\times\GL_n(\k)$ on $E$ and hence an action of  $G=(\GL_{n-1}(\k)\times\GL_n(\k))/\k^*$ on $\P E$.
Finally, since the stabilizer of $\Phi\in \bb V^s$ under the action of $G$ is trivial, it acts trivially on the fibres of $\P E$ and thus $\P E$ descends to a projective $\P_{3n+2}$-bundle
\[
\bb B\xra{\nu} N=N(3; n-1, n)=\bb V^s/G.
\]

Let $\bb W$ be the affine variety of morphisms
\begin{equation}\label{eq: res0}
\O_{\P_2}(-3)\oplus (n-1)\O_{\P_2}(-2) \xra{A} n \O_{\P_2}(-1).
\end{equation}

Notice that $\bb W$ can be identified with
$\bb V\times \bb U_2$ by the isomorphism
\[
\bb V\times \bb U_2\ra \bb W,\quad (\Phi, Q)\mto \pmat{Q\\\Phi}.
\]
The group
\[
G'=(\Aut( \O_{\P_2}(-3)\oplus (n-1)\O_{\P_2}(-2) )\times\Aut( n \O_{\P_2}(-1) )/\k^*.
\]
acts on $\bb W\iso \bb V\times \bb U_2$. As shown in~\cite[Proposition~7.7]{MaicanTwoSemiSt}, $\bb B$ is a geometric  quotient of $\bb V^s\times \bb U_2\setminus \Im F$ with respect to $G'$.

\subsection{Moduli space $M_{dm-1}(\P_2)$}\label{subsection: moduli space}
Let $d\ge 4$, $d=n+1$, be an integer. Let $M=M_{dm-1}(\P_2)$ be the Simpson moduli space of (semi-)stable sheaves on $\P_2$ with Hilbert polynomial $dm-1$. In~\cite{MaicanTwoSemiSt} is was shown that $M$ contains an open dense subvariety $M_0$
 of isomorphism classes of sheaves $\cal F$ with $h^0(\cal F)=0$.

\subsubsection{Sheaves in $M_0$}
By~\cite[Claim~4.2]{MaicanTwoSemiSt} the sheaves in $M$ without global sections are exactly the cokernels of the injective morphisms~\refeq{eq: res0} with $A=\smat{Q\\\Phi}$, $\det A\neq 0$,  $\Phi\in \bb V^s$. This allows to describe $M_0$ as an open subvariety in $\bb B$.

Let $\bb W_0$ be the open subvariety of injective morphisms in $\bb W$ parameterizing the points from $M_0$.
Since the determinant of a matrix from the image of $F$ is zero, one sees that
\[
\bb W_0\subset \bb V^s\times \bb U_2\setminus \Im F,
\]
which allows to conclude that $M_0=\bb W_0/G'$ is an open subvariety of $\bb B$.

Let $\bb B_0=\bb B|_{N_0}$ be the restriction of $\bb B$ to the open subscheme $N_0\subset N$.

From the exact sequence~\refeq{eq:sequence for points ideal} it follows that a matrix $A=\smat{Q\\\Phi}\in \bb W$, $\Phi\in \bb V_0$, has zero determinant if and only if $A$ lies in the image of $F$. Therefore, the fibres of $\bb B$ over $N_0$ are contained in  $M_0$ and thus $\bb B_0\subset M_0$. As shown in~\cite{Yuan} the codimension of the complement of $\bb B_0$ in $M$ is at least $2$.

\subsubsection{Sheaves in $\bb B_0$}\label{subsubsection: over B_0}
Sheaves in $\bb B_0$ are exactly the twisted ideal sheaves $\cal I_{Z\subset C}(d-3)$ of zero dimensional schemes $Z$ of length $l$ lying on a curve $C$ of degree $d$ such that $Z$ is not contained in a curve of degree $d-3$.
In other words the sheaves $\cal F$ in $\bb B_0$ are given by the short exact sequences
\[
0\ra \cal F\ra \O_C(d-3)\ra \O_Z\ra 0
\]
with $Z\subset C$ as described.

A fibre over a point $[\Phi]$ from  $N_0$ can be seen as the space of plane curves of degree $d$ through the corresponding subscheme of $l$ points. The identification is given by the map
\[
\nu^{-1}([\Phi]) \ni [A]\mto \point{\det A}\in \P(S^dV^*).
\]
Indeed, if two matrices over $[\Phi]$  are equivalent, then their determinants are equal up to a non-zero constant multiple and hence the map above is well-defined. On the other hand, if two matrices $A=\smat{Q\\ \Phi}$ and $A'=\smat{Q'\\\Phi}$ have equal determinants, then $Q-Q'$ lies in the syzygy module of the maximal minors of $\Phi$ and hence (cf.~\refeq{eq:sequence for points ideal}) is a linear combination of the rows of $\Phi$, which means that $A$ and $A'$ are equivalent.

\section{Singular sheaves}\label{section: Singular sheaves}
Let $M'$ be the subvariety of singular sheaves in $M$ and  let $M_0'=M'\cap M_0$. Let us consider the restriction of $\nu$ to $M_0'$ and describe some of its fibres.

\subsection{Generic fibres}\label{subsection: generic fibres}
Let $N_c$ be the open subset of $N_0$ that corresponds to $l$ different points.
Under the isomorphism $N_0\iso H_0$ it corresponds to the open  subvariety $H_c\subset H_0$ of the  configurations of $l$ points on $\P_2$ that do not lie on a curve of degree $d-3$.

Let $[\Phi]\in N_c$ and let $Z=\{\pt_1,\dots, \pt_l\}$ be the corresponding zero-dimensional scheme.

Let $m_1, \dots, m_l$ be the monomials of degree $d-3$ in variables $x_0, x_1, x_2$ ordered, say, in lexicographical order:
\[
m_1=x_0^{d-3}, \quad m_2=x_0^{d-4}x_1, \quad m_3=x_0^{d-4}x_2,\quad \dots
\]
Since $Z$ does not lie on a curve of degree $d-3$, the matrix
\[
\pmat{
m_1(\pt_1)&\dots&m_l(\pt_1)\\
\vdots&\ddots&\vdots\\
m_1(\pt_l)&\dots&m_l(\pt_l)
}
\]
has full rank.
Assume without loss of generality that $\pt_1=\point{1,0, 0}$.
Then the matrix is
\[
\pmat{
1&0&\dots&0\\
m_1(\pt_2)&m_2(\pt_2)&\dots&m_l(\pt_2)&\\
\vdots&\vdots&\ddots&\vdots\\
m_1(\pt_l)&m_2(\pt_l)&\dots&m_l(\pt_l)&
}
\]
and therefore the matrix
\[
\pmat{
m_1(\pt_2)&m_2(\pt_2)&\dots&m_l(\pt_2)&\\
\vdots&\vdots&\ddots&\vdots\\
m_1(\pt_l)&m_2(\pt_l)&\dots&m_l(\pt_l)&
}
\]
has full rank and thus there exists a homogeneous polynomial $q$ of degree $d-3$ vanishing at the points $\pt_2,\dots, \pt_l$ with the coefficient $1$ in front of the monomial $m_1=x_0^{d-3}$. Therefore, the forms $x_0^2x_1q$ and $x_0^2x_2q$ vanish at $Z$. Notice that $x_0^2x_1q$ has the monomial $x_0^{d-1}x_1$ but does not have
$x_0^{d-1}x_2$, and $x_0^2x_2q$ has $x_0^{d-1}x_2$ but does not have
$x_0^{d-1}x_1$.

Let $\cal E$ be a sheaf over $[\Phi]$. By Lemma~\ref{lemma:ideal point} it is singular at $\pt_1$ if and only if $\pt_1$ is a singular point of the support $C$ of $\cal E$. The latter holds if and only if
in the homogeneous polynomial of degree $d$ defining $C$
the coefficients
in front of the monomials $x_0^{d-1}x_1$ and $x_0^{d-1}x_2$ vanish.
Therefore, taking into account the considerations above,
sheaves over $[\Phi]$ singular at $\pt_1$ constitute a projective subspace of codimension $2$ in the fibre $F\defeq\nu^{-1}([\Phi])\iso \P_{3d-1}$.
Since our argument can be repeated for each point $\pt_i$, $i=1,\dots, l$, we conclude that the  sheaves over $[\Phi]$ singular at $\pt_i$ constitute a projective subspace $F_i$ of codimension $2$ in the fibre $\nu^{-1}([\Phi])\iso \P_{3d-1}$.

Now let us clarify how the linear subspaces in the fibre corresponding to different points $\pt_i$ intersect with each other.
First of all notice that $Z$ must contain a triple of non-collinear points because otherwise $Z$ must lie on a line.
Therefore, in addition to the assumption $\pt_1=\point{1,0,0}$, we can assume without loss of generality that $\pt_2=\point{0,1,0}$, $\pt_3=\point{0,0,1}$.
Then the conditions for being singular at these three different points read as the absence of the following monomials in the equation of $C$:
\begin{align}\label{eq: monomials}
\begin{split}
x_0^{d-1}x_1, x_0^{d-1}x_2 \quad \text{for the point $\pt_1$};\\
x_1^{d-1}x_0, x_1^{d-1}x_2 \quad \text{for the point $\pt_2$};\\
x_2^{d-1}x_0, x_2^{d-1}x_1 \quad \text{for the point $\pt_3$}.
\end{split}
\end{align}
These conditions are clearly independent of each other. Moreover, these conditions are independent from the conditions imposed on the curves of degree $d$ by the requirement $Z\subset C$. This is the case since the forms $x_0^2x_1q$ and $x_0^2x_2q$ constructed above clearly do not have monomials
$x_1^{d-1}x_0$, $x_1^{d-1}x_2$,
$x_2^{d-1}x_0$, $x_2^{d-1}x_1$. This produces a way to construct the curves through $Z$ with only one of the monomials from~\refeq{eq: monomials}.

We conclude that for every pair of different indices
\(
\codim_F F_i\cap F_j=4.
\)
Moreover,
for every  triple  of different indices $i, j, \nu$ corresponding to $3$ non-collinear points $\pt_i$, $\pt_j$, $\pt_\nu$ from $Z$
\(
\codim_F F_i\cap F_j\cap F_\nu=6.
\)
Finally we obtain the following.
\begin{lemma}
The fibres of $M'_0$ over $N_c$ are unions of $l$ different linear subspaces of $\nu^{-1}([\Phi])\iso \P_{3d-1}$ of codimension $2$
such that each pair intersects in codimension $4$ and each triple corresponding to $3$ non-collinear points intersects in codimension $6$. In particular the fibres are singular.

\end{lemma} 
\begin{rem}\label{rem: transversal intersection 3}
One can show that \(
\codim_F F_i\cap F_j\cap F_\nu=6
\)
for every  triple  of different indices $i, j, \nu$ corresponding to $3$ different points $\pt_i$, $\pt_j$, $\pt_\nu$ from $Z$.
\end{rem}
\begin{rem}
In general it is not true that $F_i$ intersect transversally. For example, for $d=6$ and $Z=\{\pt_1,\dots, \pt_{10}\}$ with
\begin{align*}
&\pt_1=\point{1,0,0}, &&\pt_2=\point{0,1,0}, &&\pt_3=\point{0,0,1}, &&\pt_4=\point{0,1,1}, &&\pt_5=\point{0,1,-1},\\
&\pt_6=\point{1,-2, 0}, &&\pt_7=\point{1,2,-1}, &&\pt_8=\point{1,1,-2}, &&\pt_9=\point{1,-1,1}, &&\pt_{10}=\point{1,1,-1},
\end{align*}
$\codim_F F_1\cap F_2\cap F_3\cap F_4=8$ but $\codim_F F_1\cap F_2\cap F_3\cap F_4\cap F_5=9$.
\end{rem}

\subsection{Fibres over $N_1$}
Let $N_1$ be the open subset of $N_0\setminus N_c$ that corresponds to $l-2$ different simple points and one  double point.

Let $[\Phi]\in N_1$ and let $Z=\{\pt_1\} \cup \{\pt_2,\dots, \pt_{l-1}\}$ be the corresponding zero-dimensional scheme, where $\pt_1$ is a double point. Without loss of generality, applying if necessary a coordinate change, we may assume that $\pt_1$ is given by the ideal $(x_1^2, x_2)$.

\subsubsection{Sheaves singular at the double point}
Since $Z$ does not lie on a curve of degree $d-3$, the matrix
\[
\pmat{
1&0&0&\dots&0\\
0&1&0&\dots&0\\
m_1(\pt_2)&m_2(\pt_2)&m_3(\pt_2)&\dots&m_l(\pt_2)\\
\vdots&\vdots&\vdots&\ddots&\vdots\\
m_1(\pt_{l-1})&m_2(\pt_{l-1})&m_3(\pt_{l-1})&\dots&m_l(\pt_{l-1})
}
\]
has  full rank. Therefore, there exist homogeneous polynomials $q$ and $q'$ of degree $d-3$ vanishing at the points $\pt_2, \dots, \pt_{l-1}$ such that
$q$ does not have  monomial $x_0^{d-4}x_1$ but has term $x_0^{d-3}$ and
$q'$ does not have  monomial $x_0^{d-3}$ but has term $x_0^{d-4}x_1$.
This implies that the forms $x_0^2x_2q$ and $x_0^2x_1q'$ vanish at $Z$.

Let $\cal E$ be a sheaf over $[\Phi]$. Let $C$ be its support. By Lemma~\ref{lemma:ideal fat curv point} one concludes that $\cal E$ is singular at $\pt_1$ if and only if
in the homogeneous polynomial of degree $d$ defining $C$
the coefficients
in front of the monomials $x_0^{d-2}x_1^2$ and $x_0^{d-1}x_2$ vanish.
Therefore, taking into account the considerations above, sheaves over $[\Phi]$ singular at $\pt_1$ constitute a projective subspace of codimension $2$ in the fibre $\nu^{-1}([\Phi])\iso \P_{3d-1}$.

\subsubsection{Sheaves singular at simple points}
Assume in addition that $\pt_2=\point{0,0,1}$. We can still do this without loss of generality because $Z$ can not lie on a line.
Then the matrix above is
\[
\pmat{
1&0&0&\dots&0\\
0&1&0&\dots&0\\
0&0&0&\dots&1\\
m_1(\pt_{3})&m_2(\pt_{3})&m_3(\pt_{3})&\dots&m_l(\pt_{3})\\
\vdots&\vdots&\vdots&\ddots&\vdots\\
m_1(\pt_{l-1})&m_2(\pt_{l-1})&m_3(\pt_{l-1})&\dots&m_l(\pt_{l-1})
}
\]
Then, as in~\ref{subsection: generic fibres}, we obtain a homogeneous polynomial $q''$ vanishing at $Z\setminus \{\pt_2\}$ with coefficient $1$ in front of the monomial $m_l=x_2^{d-3}$. Then the polynomials $x_1x_2^2q''$ and $x_0x_2^2q''$ vanish at $Z$, the former one has $x_1x_2^{d-1}$ but does not have $x_0x_2^{d-1}$, the latter one has $x_0x_2^{d-1}$ but does not have $x_1x_2^{d-1}$. This means that the sheaves over $[\Phi]$ singular at $\pt_2$ constitute a projective subspace
of codimension $2$ in the fibre $\nu^{-1}([\Phi])\iso \P_{3d-1}$.  
This also shows that the sheaves over $[\Phi]$ singular at $\pt_j$ such that $\pt_1$ and $\pt_j$ do not lie on a line constitute a projective subspace
of codimension $2$ in the fibre $\nu^{-1}([\Phi])\iso \P_{3d-1}$.

Suppose there exists a  point $\pt_j$ such that $\pt_1$ and $\pt_j$ lie on a line.
Without loss of generality we can assume that in this case $\pt_j=\point{0,1,0}$.
Then we can construct a homogeneous polynomial $q'''$ of degree $d-3$ through $Z\setminus\{\pt_j\}$ with coefficient $1$ in front of the monomial $x_1^{d-3}$.
Then the polynomials $x_0x_1^2q'''$ and $x_1^2x_2q'''$ vanish at $Z$,
the former one has $x_0x_1^{d-1}$ but does not have $x_1^{d-1}x_2$, the latter one has  $x_1^{d-1}x_2$ but does not have $x_0x_1^{d-1}$. This means that
the sheaves over $[\Phi]$ singular at $\pt_j$ such that $\pt_1$ and $\pt_j$ lie on a line constitute a projective subspace
of codimension $2$ in the fibre $\nu^{-1}([\Phi])\iso \P_{3d-1}$. This concludes the proof of the following. 

\begin{lemma}
The fibres of $M'_0$ over $N_1$ are unions of $l-1$ different linear subspaces of $\nu^{-1}([\Phi])\iso \P_{3d-1}$ of codimension $2$.
In particular the fibres are singular.
\end{lemma}

\subsection{Main result}
Now we are able to prove Theorem~\ref{tr: main}.

\subsubsection*{Singularities}
Notice that a generic fibre of a surjective morphism of smooth varieties must be smooth.
Therefore, since a generic fibre of $M_0'$ over $N$ is singular as demonstrated in~\ref{subsection: generic fibres}, we conclude that $M_0'$ is singular. Therefore, $M'$ is a singular subvariety in $M$.

\subsubsection*{Dimension}
Since, as shown in~\cite{Yuan}, the codimension of the complement of $\bb B_0$ in $M$ is at least $2$, in order to demonstrate that $\codim_M M' = 2$, it is enough to show that $\codim_{\bb B_0} \bb B_0\cap M'=2$.
Denote $\bb B_c=\bb B|_{N_c}$, $\bb B_1=\bb B|_{N_1}$. Since the complement of $N_c\sqcup N_1$ in $N_0$ has codimension $2$, it is enough to show
\[
\codim_{\bb B_c\sqcup \bb B_1} M'\cap(\bb B_c\sqcup \bb B_1)=2.
\]
The latter follows immediately since the codimension of fibres of $M_0'$ over $N_c\sqcup N_1$ is $2$. This concludes the proof.

\subsubsection*{Smooth locus of $M'$}
\begin{pr}\label{pr: generic smooth}
The smooth locus of $M'$ over $N_c$ coincides with the locus of sheaves
corresponding to $Z\subset C$ such that only one of the points in $Z$ is a singular point of $C$.
\end{pr}
\begin{proof}
Notice that $H_c\iso N_c$ can be seen as an open subscheme in $S^l\P_2$.
Taking the composition of a local section (in analytic or \'etale topology) of the quotient $\prod_{1}^{l}\P_2\ra S^{l}\P_2$ with the  projection $\prod_{1}^{l}\P_2\ra\P_2$ to the $j$-th factor, we get, locally around a given $Z_0\in H_c$,  $l$ different local choices
\[
N_c\iso H_c\supset U\xra{p_{j}} \P_2, \quad j=1,\dots, l,
\]
of one point in $Z\in U\subset H_c$. Shrinking $U$ if necessary we can assume that $\bb B\ra N_c$ is trivial over $U$. Then by~\ref{subsection: generic fibres} the subvariety  $S_{j}\subset \bb B|_U$ of those sheaves given by $Z\subset C$ that are singular at point $p_{j}(Z)\in \P_2$ is isomorphic to a product of $U$ with a linear subspace of $\P_{3d-1}$ of codimension $2$ (i.~e., with $\P_{3d-3}$). Therefore, $S_{j}$  is  smooth. Notice that $M'\cap \bb B|_U$ is isomorphic to the union of $S_{j}$, $j=1,\dots, l$. Therefore,
\(
\bigcup_{j} S_{j}\setminus \bigcup_{j\neq i} S_{j}\cap S_{i}
\)
is smooth, which proves the required statement.
\end{proof}

\section{Modifying the boundary by vector bundles}\label{section: modifications}
\subsection{Normal spaces at $M'\cap\bb B_c$}
Let $[\cal F]\in \bb B_c\cap M'$ be the isomorphism class of a singular sheaf represented by a curve $C$ of degree $d$ and a configuration of $l$ points $Z\subset C$, $Z=\{\pt_1, \dots, \pt_l\}$. Assume without loss of generality that $Z\cap \Sing C=\{\pt_1,\dots, \pt_r\}\defeq Z'$, $0<r \le l$.

Using, for every $j=1,\dots, r$, the local choice of a point $p_j$ and the morphism $M\ra \P S^d V^*$, $[\cal G]\mto \Supp \cal G$, we obtain locally around $[\cal F]$ a morphism $\bb B_c\supset U_{[\cal F]}\xra{\rho_j} \UC{d}$ from a neigbourhood of $[\cal F]$ to the universal planar curve $\UC{d}$ of degree $d$.

This induces a linear map of the tangent spaces $\Tsp_{[\cal F]}\bb B_c \ra \Tsp_{(C,\pt_j)}\UC{d}$.
Let $\UC{d}'$ denote the universal singular locus of $\UC{d}$. Since $\rho_{j}$ maps $S_j$ to $\UC{d}'$, we obtain also the induced linear map on the normal spaces
\[
N_{[\cal F]}^{(j)}\defeq\Tsp_{[\cal F]}\bb B_c /\Tsp_{[\cal F]} S_j\xra{\Tsp_{[\cal F]}(\rho_j)} \Tsp_{(C,\pt_j)}\UC{d}/\Tsp_{(C,\pt_j)}\UC{d}'\eqdef N_{(C, \pt_j)}.
\]
\begin{lemma}\label{lemma: iso on normal}
The linear map $N_{[\cal F]}^{(j)}\xra{\Tsp_{[\cal F]}(\rho_j)} N_{(C, \pt_j)}$ constructed above is an isomorphism of $2$-dimensional vector spaces.
\end{lemma}
\begin{proof}
Let $F=\nu^{-1}(\nu([\cal F]))$ be the fibre of $[\cal F]$.
As already noticed $F$ can be seen as the space of curves of degree $d$ through $Z$.
For a fixed $j\in\{1,\dots, r\}$
let $F_j'$ be the fibre of $S_j$ over $\nu([\cal F])$, which can be seen as the subspace of $F$ of those curves through $Z$ singular at $\pt_j$.
Let $F_{\pt_j}$ be the space of curves of degree $d$ through $\pt_j$ and let $F'_{\pt_j}$ be its subspace of curves singular at $\pt_j$.

Our analysis in~\ref{subsection: generic fibres} implies that $\rho_j$ induces an isomorphism $\Tsp_C F/\Tsp_C F'_j\iso \Tsp_C F_{\pt_j}/\Tsp_C F'_{\pt_j}$, which concludes the proof because both $S_j$ and $\UC{d}'$ are locally trivial over $N$ and $\P_2$ respectively and hence
$N_{[\cal F]}^{(j)}\iso \Tsp_C F/\Tsp_C F'_j$,
$N_{(C, \pt_j)}\iso\Tsp_C F_{\pt_j}/\Tsp_C F'_{\pt_j}$.
\end{proof}


\subsection{$R$-bundles}

Let $\fcal U$ denote the universal family on $M\times \P_2$. 
Consider  a germ of a morphism $\gamma$ of a smooth curve $T$ to $\bb B_c$ mapping $0\in T$ to $\gamma(0)=[\cal F]$.
Let $\fcal F$ be the pullback of  $\fcal U$ along $\gamma\times \id_{\P_2}$.

If $\gamma$ is not tangent to $S_j$ at $[\cal F]$,
then $\fcal F$  represents $[\cal F]$ as a flat $1$-parameter degeneration of sheaves $[\fcal F_t]=\gamma(t)$ non-singular at point $p_{j}(\nu\circ \gamma(t))$, where $\fcal F_t=\fcal F|_{\{t\}\times\P_2}$.
If $\gamma$ is not tangent to $S_j$ at $[\cal F]$ for all $j=1,\dots, r$, then $\fcal F$ is a degeneration of non-singular sheaves to $\cal F=\fcal F_0$.

Let $\tilde{T\times \P_2}\xra{\sigma}T\times \P_2$ be the blow-up $\Bl_{\{0\}\times Z'} (T\times \P_2)$ and let $D_1=D_1(Z')$ be its exceptional divisor, which is a disjoint union of projective planes $D_1(p_j)\iso \P_2$, $j=1, \dots, r$.

The fibre of the flat morphism
\(
\tilde{T\times \P_2}\xra{\sigma}T\times \P_2\xra{pr_1} T
\)
over $0$ is a reduced surface $D(Z')=D(p_1, \dots, p_r)$ obtained by blowing-up $\P_2$ at $\{\pt_1,\dots, \pt_r\}$ and attaching the
surfaces $D_1(\pt_j)\iso \P_2$, $j=1,\dots, r$,
to $D_0(Z')=D_0(\pt_1,\dots, \pt_r)=\Bl_{Z'} \P_2$
along the exceptional lines $L_1,\dots, L_r$.

Let $\fcal E$ be the sheaf on $\tilde{T\times \P_2}$ obtained as the quotient of the pullback $\sigma^*(\fcal F)$ by the subsheaf $\cTor_{\cal I_{D_1}}(\sigma^*(\fcal F))$ generated by the sections annihilated by the ideal sheaf $\cal I_{D_1}$ of the exceptional divisor $D_1$.

\begin{lemma}
Assume that $\gamma$ is not tangent to $S_j$ at $[\cal F]$ for all $j=1,\dots, r$.
Then the sheaf $\fcal E$ is a flat family  of  $1$-dimensional sheaves. Its fibres $\fcal E_t$, $t\neq 0$, are non-singular sheaves  on $\P_2$, the fibre $\cal E=\fcal E_0$ is a $1$-dimensional non-singular sheaf on $D(Z')$.
\end{lemma}
\begin{proof}
We shall show that our definition of $\fcal E$ locally coincides with the construction from~\cite{IenaUnivCurve}.

Fix some local coordinates $x_i, y_i$ at $\pt_i$. Then in some neighbourhood $U_i$ of $\pt_i$ the sheaf $\cal F$ is just an ideal sheaf of  $\pt_i$ on $C$ and hence can be given as the cokernel of a morphism
  \[
  2\O_{U_i} \xra{A} 2\O_{U_i}, \quad A=\smat{x_i&y_i\\a_{i}&b_{i}},
  \]
  for some  polynomials $a_i, b_i$ in $x_i, y_i$.
As $\pt_i\in \Sing S$ for  $i=1,\dots, r$, $a_{i}$ and $b_{i}$ do not have constant terms.

Let $\fcal U_d$ denote the universal family on $\UC{d}\times \P_2$.
Then, for a fixed $j\in\{1,\dots, r\}$, locally around the point  $[\cal F]\times \pt_j\in \bb B_c\times \P_2$,  $\fcal U$ is isomorphic to the pullback of $\fcal U_d$ along $\rho_j\times\id_{\P_2}$.
Then $\fcal F$ is isomorphic locally around $0\times \pt_j$ to the pullback of $\fcal U_d$ along $\gamma_j\times \id_{P_2}$, where $\gamma_j=\rho_j\circ\gamma$.

The latter means that the family $\fcal F$ is given around $0\times \pt_j$ as a cokernel of the morphism
  \[
  2\O_{T\times U_j} \xra{\smat{x_j&y_j\\v_j&w_j}+t\cdot B(t)} 2\O_{T\times U_j},\quad B(t)=\smat{b_{11}(t)&b_{12}(t)\\b_{21}(t)&b_{22}(t)},
  \]
  where $b_{11}(t), b_{12}(t)$ take values in $\k$ and $b_{21}(t), b_{22}(t)$ are polynomials of degree not bigger than $d$.

 The blow-up $\tilde{T\times \P_2}$ over $0\times \pt_j \in T\times \P_2$ is locally just the blow-up $\Bl_{0\times \pt_j}T\times U_j\eqdef \tilde{T\times U_j}$, it can be seen as a subvariety in $T\times U_j\times \P_2$ given by the $(2\times 2)$-minors of the matrix $\smat{t&x_j&y_j\\u_0&u_1&u_2}$, where $u_0, u_1, u_2$ are some homogeneous coordinates of the last factor $\P_2$. The canonical section $s$ of $\O_{\tilde{T\times \P_2}}(D_1)$ is locally given by $t/u_0$, $x_j/u_1$, $y_j/u_2$.

As in~\cite{IenaUnivCurve} one  ``divides'' the pull back $\sigma^*(A+tB(t))$  by $s$ and obtains a family $\fcal E'$ of one-dimensional sheaves given by a locally free resolution
\[
0\ra 2\O_{\tilde{T\times U_j}}(D_1) \xra{\phi(A, B)} 2\O_{\tilde{T\times U_j}}\ra \fcal E'\ra0
\]
We claim that this construction coincides with taking the quotient by the subsheaf annihilated by $\cal I_{D_1}$. This follows from a diagram chasing on the following commutative diagram with exact rows and columns
\begin{equation*}
\begin{xy}
(105,20)*+{0}="35";
(35,15)*+{0}="23";
(105,10)*+{\cal K}="25";
(10,0)*+{0}="12";
(35,0)*+{ 2\O_{\tilde{T\times U_j}}}="13";
(75,0)*+{ 2\O_{\tilde{T\times U_j}}}="14";
(105,0)*+{\sigma^*\fcal F }="15";
(120,0)*+{ 0}="16";
(10,-15)*+{0}="02";
(35,-15)*+{2\O_{\tilde{T\times U_j}}(D_1)}="03";
(75,-15)*+{ 2\O_{\tilde{T\times U_j}}}="04";
(105,-15)*+{\fcal E'}="05";
(120,-15)*+{0.}="06";
(35,-25)*+{\cal K}="-13";
(105,-30)*+{0}="-15";
(35,-35)*+{0}="-23";
{\ar@{->}"12";"13"};
{\ar@{->}^-{\sigma^*(A+tB(t))}"13";"14"};
{\ar@{->}"14";"15"};
{\ar@{->}"15";"16"};
{\ar@{->}"02";"03"};
{\ar@{->}^-{\phi(A, B)}"03";"04"};
{\ar@{->}"04";"05"};
{\ar@{->}"05";"06"};
{\ar@{->}"23";"13"};
{\ar@{->}_{\smat{s&0\\0&s}}"13";"03"};
{\ar@{->}"03";"-13"};
{\ar@{->}"-13";"-23"};
{\ar@{->}"35";"25"};
{\ar@{->}"25";"15"};
{\ar@{->}"15";"05"};
{\ar@{->}"05";"-15"};
{\ar@{=}"14";"04"};
\end{xy}
\end{equation*}
More precisely, one shows that $\cal K$ is exactly the subsheaf of $\sigma^*\fcal F$ annihilated by $s$.

Now the flatness of $\fcal E$  and the local freeness of $\cal E$ on its support follow from~\cite{IenaUnivCurve} (Lemma~5.1 and the discussion before it).
\end{proof}

\begin{rem}
Notice that the fibre of $\sigma^*\fcal F$ over $0\in T$ is not a $1$-dimensional sheaf.
Its support contains $D_1$.
The subsheaf $\cTor_{\cal I_{D_1}}(\sigma^*(\fcal F))$ of $\sigma^*(\fcal F)$ is the maximal subsheaf supported completely in $D_1$.
\end{rem}

\begin{rem}
The sheaf $\fcal E_0$ depends only on the derivative of $\gamma$ at $0$, i.~e., only on the induced map $\Tsp_0 T\ra \Tsp_{[\cal F]}\bb B_0$ of tangent spaces.
\end{rem}

\begin{df}
The sheaf $\fcal E_0$ on $D(Z')$ as above is called an $R$-bundle associated to the pair $Z\subset C$, $Z'=Z\cap \Sing C$.
\end{df}

The following generalizes Definition~5.5 from~\cite{IenaUnivCurve}.
\begin{df}
 Two $R$-bundles
 $\cal E_1$ and $\cal E_2$ associated to $Z\subset C$ on $D(\pt_1,\dots, \pt_r)$ are called equivalent if there exists an automorphism $\phi$
 of $D(\pt_1,\dots, \pt_r)$ that acts identically on the surface $D_0(\pt_1,\dots, \pt_r)$ such that $\phi^*(\cal E_1)\iso \cal E_2$.
\end{df}

\begin{tr}\label{tr: blow-up modification}
Assume that the components $F_1,\dots, F_r$ intersect transversally at $[\cal F]$. Then the equivalence classes of $R$-bundles associated to $Z\subset C$ are in one-to-one correspondence with the points of the product of projective lines $\prod_{j=1}^r\P N^{(j)}_{[\cal F]}$.
\end{tr}
\begin{proof}
The classes of $R$-bundles around  $D_1(\pt_j)$ are parameterized by $\P N_{(C, \pt_j)}\iso \P N^{(j)}_{[\cal F]}$ by~\cite[Proposition~5.6]{IenaUnivCurve} and Lemma~\ref{lemma: iso on normal}.
Since the components $F_1,\dots, F_r$ intersect transversally, the map
\[
\Tsp_{[\cal F]}\bb B_c\setminus \cup_{j=1}^r \Tsp_{[\cal F]} S_j\xra{}\prod_{j=1}^r\P N^{(j)}_{[\cal F]}
\]
is surjective, which means that the class of an $R$-bundle around a given exceptional plane $D_1(\pt_j)$ is independent of the class of the $R$-bundle around other planes $D_1(\pt_i)$, $i\neq j$. This concludes the proof.
\end{proof}

\begin{rem}
The assumption of Theorem~\ref{tr: blow-up modification} is always satisfied at least for $r\le 3$ by Remark~\ref{rem: transversal intersection 3}.
In particular it is the case for $d=4$. Remark~\ref{rem: transversal intersection 3} also implies that the locus of $[\cal F]\in \bb B_c$ that do not satisfy the assumption on transversality lies at least in codimension $7$.
\end{rem}

\begin{rem}
All possible $R$-bundles can be produced simultaneously as fibres of a family of sheaves over (an open subset of) $\tilde M\defeq \Bl_{M'} M$. 
Indeed, pull back the universal family over $M$ to a family $\fcal{U}_{\tilde M}$ over $\tilde M$.
Let 
$\tilde{\tilde M\times \P_2}\xra{}\tilde M\times \P_2$ be the blow up along the subvariety in $\tilde M\times \P_2$ where $\fcal{U}_{\tilde M}$ is singular, pull $\fcal{U}_{\tilde M}$ back and consider the quotient  $\tilde{\fcal U}$ by the subsheaf annihilated by the ideal sheaf of the exceptional divisor of this blow-up.

Let $\bb B_{gen}$ denote the open locus of those $[\cal F]\in \bb B_0$ satisfying the conditions of Theorem~\ref{tr: blow-up modification}. 
Then $\tilde{\bb B}_{gen}\defeq \Bl_{M'\cap \bb B_{gen}} \bb B_{gen}$ is smooth  by
Theorem~1.2 and Theorem~1.3 from~\cite{LiWonderful}.
The fibre of the exceptional divisor over $[\cal F]\in \bb B_{gen}$ coincides with the product  $\prod_{j=1}^r\P N^{(j)}_{[\cal F]}$. The restriction of 
$\tilde{\fcal U}$ to $\tilde{\bb B}_{gen}$ is a flat family of $1$-dimensional non-singular sheaves: non-singular $(dm-1)$-sheaves together with $R$-bundles associated to $Z\subset C$, $Z\cap \Sing C\neq \emptyset$.

This allows us to see the blow-up $\tilde{\bb B}_{gen}$ as a process that substitutes the boundary $M'\cap \bb B_{gen}$ by a divisor consisting of non-singular $1$-dimensional sheaves, which generalizes the construction from~\cite{IenaUnivCurve}.
\end{rem}

\subsection{Further work.}
A more detailed discussion of the relation of the blow-up $\Bl_{M'}M$ with modifying the boundary of $M$ by vector bundles should follow in a separate paper.

\appendix
\section{On the ideals of points on planar curves}\label{section: commutative algebra}
Let $R=\O_{C, p}$ be a local $\k$-algebra of a curve $C$ at point $p\in C$.
Let $I\subset R$ be an ideal of $R$. As a submodule of a free module, $I$ is a torsion free $R$-module. If $R$ is regular, i.~e., if $p$ is a smooth point of $C$, then $I$ is free. Therefore, the non-regularity of $R$ is a necessary condition for the non-freeness of $I$.

\subsection{Ideals of simple points on a curve}
Let $\frak m=\frak m_{C, p}$ be the maximal ideal of $R$ and let
 $\k_p=R/\frak m$ be the local ring of the structure sheaf of the one point subscheme $\{p\}\subset C$.


\begin{lemma}\label{lemma:ideal point}
Consider an  exact sequence of $R$-modules.
\[
0\ra \frak m\ra R\ra \k_p\ra 0
\]
with a non-zero $R$-module $\frak m$.
Then $\frak m$ is free if and only if $R$ is regular.
\end{lemma}
\begin{proof}
If $\frak m$ is free, then  $\frak m\iso R$ (otherwise $\frak m\ra R$ would not be injective)  and we obtain an exact sequence of $R$-modules
\[
0\ra R\ra R\ra \k_p\ra 0,
\]
which
means  that the maximal ideal $\frak m$ of $p$  is generated by one element. Therefore,  $R$ is regular in this case.

Vice versa, assume  $R$ is regular. Notice that $\frak m$ is always a torsion free $R$-module as a submodule of $R$. Therefore, if $R$ is regular, $\frak m$ is free as a torsion free module over a regular one-dimensional local ring.
\end{proof}

\subsection{Ideals of fat curvilinear points on a planar curve}

Assume that $C$ is a planar curve locally defined as the zero locus of $f\in\k[x, y]$. Assume  $p=0$, consider  a  fat curvilinear point $Z$ at $p$ given by the ideal
\[
(x-h(y), y^n)\subset \k[x, y], \quad h(y\in \k[y]),\quad  h(0)=0,\quad \deg h < n,
\]
and assume that
 $Z$ is  a subscheme of $C$. Let $I\subset R$ be its ideal.

Since we assumed $Z\subset C$, one can write
\[
f=\det\pmat{x-h(y)&y^n\\u(y)&v(x,y)}, \quad u(y)\in \k[y], \quad v(x, y)\in \k[x,y].
\]

\begin{lemma}\label{lemma:ideal fat curv point}
Keeping the notations as above, let $R$ be a non-regular ring, i.~e., let $p$ be a singular point of $C$.
Then  $I$ is non-free if and only if $u(0)=0$.

If $I$ is free, then it is generated by $x-h(y)$ and there is an isomorphism
\[
R \iso I,\quad  r\mto r\cdot(x-h(y)).
\]
\end{lemma}
\begin{proof}
First of all we shall show that $I$ is generated by one element if and only if $u(0)\neq 0$.

Since $R=\k[x, y]_{(x, y)}/(f)$, it is enough to answer the question when the ideal
\[
(x-h(y), y^n, f)\subset \k[x, y]_{(x, y)}
\]
  equals $(\xi, f)$ for some $\xi\in \k[x, y]_{(x, y)}$.

\noindent``$\Leftarrow$'': If $u(0)\neq 0$, then $(x-h(y), y^n, f)=(x-h(y), f)$.

\noindent``$\Rightarrow$'': Let  $(x-h(y), y^n, f)=(\xi, f)$,  without loss of generality we can assume that
\[
\xi\in(x-h(y), y^n)
\]
and
\[
\xi=a\cdot(x-h(y))+ b\cdot y^n
\]
for some $a, b\in \k[x, y]_{(x, y)}$.
In order to show that  $u(0)\neq 0$, we suppose that the contrary holds true, i.~e., $u(0)=0$.

Since one can embed $\k[x, y]_{(x, y)}$ into the ring of formal power series $\k[[x, y]]$, we are going to consider the elements of $\k[x, y]_{(x, y)}$ as power series in $x, y$.

Since $x-h(y)\in (\xi, f)$, then $x-h(y)=c\cdot \xi + d\cdot f=ca(x-h(y))+cby^n+df$. As the orders of $cby^n$ and $df$ are at least $2$, we conclude that $a$ and $c$ are units. We can assume without loss of generality that
\[
\xi=x-h(y)+\eta(x, y), \quad \ord \eta(h(y), y)\ge n.
\]
As $y^n\in (\xi, f)$, it must hold $y^n=C\xi+Df$ for some $C$ and $D$. Since by our assumption  $u(0)=0$, evaluating this equality at $x=h(y)$, we conclude that
\[
y^n=C(h(y), y)\cdot \eta(h(y), y)+D(h(y), y)\cdot (-u(y)\cdot y^n).
\]
Therefore, $C$ must be a unit and  $\ord \xi(h(y),y)=\ord \eta(h(y), y)=n$.

Substituting $x$ by $h(y)$ in the equality $x-h(y)=c\cdot \xi + d\cdot f$, we get
\[
0=c(h(y), y)\cdot \eta(h(y),y) + d(h(y), y)\cdot (-u(y)\cdot y^n).
\]
Since $c$ is a unit, it contains a non-zero constant term and hence the product  $c(h(y), y)\cdot \eta(h(y),y)$ has order $n$. On the other hand, since $u(0)=0$ by our assumption, the order of
$d(h(y), y)\cdot u(y)y^n$ is at least $n+1$. We obtain a contradiction, which shows that our assumption was wrong.

Notice that if $I$ is free, then it must be one-generated. On the other hand we see that if $I$ is one-generated, then it is generated by $x-h(y)$. In this case $u(0)\neq 0$ and therefore $f$ is not divisible by $x-h(y)$. Thus $x-h(y)$ is not a zero divisor in $R$ and there is an isomorphism $R\iso I$, $r\mto rx$. The latter means that $I$ is one-generated if and only if it is free, which concludes the proof.
\end{proof}

\def\cprime{$'$} \def\cprime{$'$} \def\cprime{$'$}

\end{document}